\documentclass[12pt]{amsart}

\usepackage[margin=1.15in]{geometry}

\usepackage{amsmath,amscd,amssymb,amsfonts,latexsym,wasysym, mathrsfs, mathtools,hhline,color, bm}
\usepackage[all, cmtip]{xy}
\usepackage{comment}

\usepackage{url,mathtools,amsmath}

\definecolor{hot}{RGB}{65,105,225}

\usepackage[colorlinks=true, linkcolor=hot ,  citecolor=hot, urlcolor=hot]{hyperref}

\usepackage{ textcomp }
\usepackage{ tipa }

\usepackage[normalem]{ulem}
\usepackage{graphicx,enumerate}

\usepackage{enumitem}

\theoremstyle{plain}
\newtheorem{theorem}{Theorem}[section]
\newtheorem{proposition}[theorem]{Proposition}

\newtheorem{corollary}[theorem]{Corollary}

\newtheorem{lemma}[theorem]{Lemma}

\theoremstyle{definition}
\newtheorem{definition}[theorem]{\sc Definition}

\numberwithin{equation}{section}

\newcommand\hot{\mathrm{h.o.t.}}
\newcommand\sS{\mathscr{S}}
\newcommand\cO{\mathcal{O}}

\newcommand\sW{\mathscr{W}}
\newcommand\sZ{\mathscr{Z}}
\newcommand\bx{\mathbf{x}}
\newcommand\ity{\infty}

\def\bR{\mathbb{R}}
\def\bC{\mathbb{C}}
\def\bP{\mathbb{P}}
\def\bX{\mathbb{X}}

\def\m{\setminus}

\newcommand{\fin}{\hspace*{\fill}$\square$}

\renewcommand{\d}{{\mathrm d}}

\DeclareMathOperator{\Sing}{Sing}                                    
\DeclareMathOperator{\Jac}{Jac}
\DeclareMathOperator{\mult}{mult}
\DeclareMathOperator{\ord}{ord} 
 \DeclareMathOperator{\grad}{grad}              
\DeclareMathOperator{\reg}{reg} 
\DeclareMathOperator{\rank}{rank}

\DeclareMathOperator{\Pol}{Pol}
\DeclareMathOperator{\cl}{closure}

\title[Morse numbers of complex polynomials]{Morse numbers of complex polynomials}

\author{Lauren\c{t}iu Maxim}
\address{L. Maxim: Department of Mathematics,  University of Wisconsin-Madison,  480 Lincoln Drive, Madison WI 53706-1388, USA, \newline
{\text and} Institute of Mathematics of the Romanian Academy, P.O. Box 1-764, 70700 Bucharest, ROMANIA.}
\email {maxim@math.wisc.edu}

\author[M. Tib\u{a}r]{Mihai Tib\u{a}r}
\address{M. Tib\u{a}r: Universit\' e de  Lille, CNRS, UMR 8524 -- Laboratoire Paul Painlev\'e, F-59000 Lille, France}  
\email {mihai-marius.tibar@univ-lille.fr}


\begin{document}


\begin{abstract}  
To a  complex polynomial function $f$ with arbitrary singularities  we associate the number of Morse points in a general linear Morsification $f_{t} :=  f - t\ell$.   We produce computable algebraic formulas in terms of invariants of $f$ for the numbers of stratwise Morse trajectories which abut, as $t\to 0$,  to some point of the space or at infinity.
\end{abstract}

\maketitle


\section{Introduction}

Finding the number of Morse points of a function is a classical topic of fundamental  interest and with many applications,  in real and in complex geometry. This problem occurs for instance in Lusternik-Schnirelmann category, Floer homology, maximum likelihood degree, etc.
We consider here Morsifications $f_{t}$ of a given complex polynomial function $f := f_{0}$ on a possibly singular complex affine variety $X\subset \bC^{N}$. We address the problem of effectively computing the number of Morse points  in case of linear Morsifications $f_{t}= f- t\ell$, and we show more precisely how these Morse points abut in clusters whenever the parameter $t$ converges to 0.  

 In the local setting, the case of holomorphic functions with an \emph{isolated singularity} $g\colon (\bC^{n+1}, 0) \to (\bC, 0)$ is classical,   extensively studied for instance in Milnor's book \cite{Mi}. Brieskorn showed in \cite[Appendix]{Bri} that in any Morsification of $g$, i.e., in a continuous family  $g_{t}$ of holomorphic functions, with $g_{0} = g$, for $t\not= 0$ close enough to the origin, the function $g_{t}$  has  a precise number of $\mu(g)$  Morse points which converge to the origin as $t\to 0$, where $\mu(g)$ denotes the \emph{Milnor number of $g$} at 0 and has topological and algebraic interpretations.  This result follows from the conservation of topology of the general fibre of $g_{t}$ in the deformation, and  extends to the setting of a stratified isolated singularity, see, e.g.,  \cite{Ti-bouquet}, \cite{MT2}.  Whenever the function $g$ has a nonisolated singularity,  Morsifications still exist but the conservation of topology does not hold anymore: as   $t\to 0$, the isolated Morse points of $g_{t}$ explode into the nonisolated singularity of $g_{0}=g$. An important problem in this context is to interpret the number of Morse points of the Morsification $g_{t}$ in terms of the topology of $g$.

The global setting of any non-constant polynomial function $f\colon \bC^{n+1} \to \bC$ is even more subtle.
First of all, the number of Morse points depends on the type of deformation\footnote{The number of Morse points in some Morsification of constant degree is bounded from above by $(d-1)^{n+1}$, where $d = \deg f$, see \cite{ST-defo}. This upper bound is realized for instance in deformations of type $f_{t} = f - th$ with $h$ a general homogeneous polynomial of degree $d$, cf \cite{ST-defo}.}.
If we stick to linear deformations, namely  $f_{t} := f -t\ell$, for some linear function $\ell$, then
it turns out that the number of Morse points of $f_{t}$
is determined by $f$  as soon as $\ell$ in chosen in a Zariski-open set of \emph{general linear functions}, and the parameter $t$ is close enough to $0\in \bC$.  Such a deformation will be called a \emph{linear Morsification of $f$}. 

The natural question:

\smallskip

\noindent  ``\emph{how many singular points are there in a linear Morsification $f_{t} = f - t\ell$ of a given polynomial $f$ and how to compute this number?}'' 

\smallskip
 
\noindent has been asked more recently in \cite{MRW2, MT1} in relation to the problem of estimating the Euclidean distance (ED) degree of (the complexification of) an algebraic statistical model. Let us recall that the ED degree ${\rm EDdeg}(X)$ of an irreducible affine variety $X \subset \bC^N$ counts the number of critical points of the square distance function $d_u(x):=\sum_{i=1}^N (x_i-u_i)^2$ on the smooth locus $X_{\reg}$ of $X$, for a {\it generic} data point $u \in \bC^N$, see \cite{DHOST}. When $X$ is the complexification of a real algebraic model $X_\bR$, ${\rm EDdeg}(X)$ measures the complexity of the nearest point problem for $X_\bR$. A purely topological interpretation of the ED degree, as a weighted Euler characteristic, was given in \cite{MRW1} by using results on the Euler obstruction. The motivation for considering a linear Morsification of $d_u$ was to understand such a nearest point problem for a data point $u \in \bC^N$ which is {\it non-generic} (e.g., it belongs to what is called the ED \emph{discriminant}). Results of \cite{MRW2}, see also \cite{MT1}, allow us to handle such situations by observing the limiting behavior of critical sets obtained for generic choices of data. More precisely, by adding some ``noise'' ${\epsilon} \in \bC^N$ to an arbitrary data point $u \in \bC^N$, one is back in the generic situation, and the limiting behavior of critical points of $d_{{u}+t {\epsilon}}$ on $X_{\rm reg}$ for $t \in \bC^*$ (with $\vert t \vert$ very small), as $t$ approaches the origin of $\bC$, yields valuable information about the initial nearest point problem. 
Indeed, one can write in this case $d_{{u}+t {\epsilon}}(x)=d_{u}(x)- t\ell(x) + c,$
with $\ell(x)=2 \sum_{i=1}^n  \epsilon_i  x_i$ and $c$ is a constant with respect to $x$.
So the critical points of $d_{{u}+t {\epsilon}}$ coincide with those of $d_{u}-t\ell$. 
As the genericity of   ${\epsilon}$ is equivalent to that of $\ell$,  the number of critical points of $d_{{u}+t {\epsilon}}$, and hence of $d_{u}-t\ell$, on $X_{\rm reg}$, for $t \in \bC^*$ very close to 0,  computes the ED degree  ${\rm EDdeg}(X)$ of $X$. 

 In the general framework of any polynomial function $f\colon X\to \bC$ on a singular irreducible affine variety $X\subset \bC^{N}$,  the paper  \cite{MRW2} gives formulas for the total number $\# \Sing (f_{t})_{|X_{\reg}}$ of Morse points on the regular part of $X$ in a linear Morsification $f_{t}$ under the condition that no Morse point escapes to infinity as $t\to 0$. 
It has  been found in \cite{MRW2, MT1} 
that the trajectories of the Morse singularities of $f_{t}$ converge to certain (finitely many)  points of the singular locus of $f$, and their places depend on the the choice of $\ell$.
The total number $n_{V}$ of Morse  singularities collapsing to  points on  strata $V\subset \Sing f$ are computed in  \cite{MRW2} in terms of the vanishing cycles of $f$, starting from 
 an Euler obstruction interpretation. Some more explicit formulas for $n_{V}$ have been obtained in \cite{MT1}.

 \smallskip
 
The new aim of our study is to give geometric and effectively computable formulas for the number of Morse points
which abut at each point, that we shall call here \emph{attractor}. This local number will be called  \emph{Morse index of an attractor}.  
 
Finding the Morse indices at all the attractors of the space, that we complete in Section \ref{s:infinity}, is nevertheless not the end of the story. One also observes the phenomenon of \emph{singularities escaping at infinity} when $t\to 0$. This has been only pointed out in \cite{MRW2, MT1}, and to some extent in \cite{ST-defo}, \cite{ST-exch}, \cite{ST-betti}, but remained mysterious until now.  
To get a taste of this phenomenon of Morse singularities of $f_{t}$ escaping to infinity, one may consider the following simple polynomial function  $f(x,y) = x+x^{2}y$ which has no singularity at all. However  its general linear Morsification  $f_{t} = f - t\ell$  has two Morse singularities, and both escape to infinity as $t\to 0$ -- see Example \ref{ex:classic} for more details.  

Even more surprisingly, we show that there are two ways in which 
Morse points can escape to infinity as $t\to 0$: their trajectories  may converge asymptotically to a certain fibre (like in the above-mentioned example, where this fibre is  $f^{-1}(0)$), or not, as Example \ref{ex:lau-s} shows.  As a matter of fact, there are polynomial functions $f$ such that the general linear Morsification has Morse points escaping to infinity, and such that no Morse trajectories are asymptotic to fibres of $f$ -- see  Example \ref{ex:lau-s-more} for such an instance.

 \medskip

We develop a new theoretical way of  counting the  total number of Morse points of $f_{t}:= f-t\ell$ on each stratum $V$ of the stratification $\sW$ of $X$ by an effective computation of the local Morse indices at each individual attractor.  Our procedure manages in particular the Morse trajectories at infinity and the Morse indices at ``attractors at infinity''.

We will actually give a geometric interpretation for  the Morse trajectories in all the strata of $X$  in terms of the \emph{affine relative polar curve},   as well as for the Morse indices of attractors, with an emphasis on those which appear at infinity.  
For the precise statements we refer to Theorem \ref{t:main2} and Corollary \ref{c:main2}.

\medskip 

\noindent {\bf Acknowlegments.}  Special thanks to Lauren\c tiu P\u aunescu and Cezar Joi\c ta for their help in finding and computing examples  for the less intuitive behavior at infinity. The authors acknowledge support from the project ``Singularities and Applications'' - CF 132/31.07.2023 funded by the European Union - NextGenerationEU - through Romania's National Recovery and Resilience Plan.


\section{General linear Morsifications}\label{s:linmorsif}

Let $X \subset \bC^N$ be a closed irreducible affine variety of dimension $n+1\ge 2$,  let $f\colon X \to \bC$ be the restriction to $X$ of a polynomial function $F: \bC^{N} \to \bC$, and let us assume that $f$ is non-constant.  We shall call $f$ a \emph{non-constant polynomial function on $X$}.

One may endow $X$ with a Whitney stratification $\sW$ with finitely many strata  such that its regular part $X_{\reg}$ is a stratum. We may moreover choose the coarsest such stratification (with respect to inclusion of strata), which exists, and which is called \emph{the canonical Whitney stratification of $X$}.

 Let $\Sing_{\sW}f :=\bigcup_{V\in \sW} \Sing f_{|V}$ denote the stratified singular locus of $f$ with respect to $\sW$. This is a closed set  (by the Whitney (a)-regularity condition) distributed in a finite number of fibres of $f$ which are called \emph{singular fibres}.  
  
 Let $\sS$ be a Whitney stratification of $X$ which is the coarsest refinement of the stratification $\sW$ such that $\Sing_\sW f$ is a union of strata of $\sS$.  Note that $\sS$ also has finitely many strata.
 

 
\begin{definition}[Stratified Morse function, \cite{GM}]\label{d:morsef}
We say that the restriction $h\colon X\to \bC$ of a polynomial function to $X$ is a \emph{stratified Morse function} with respect to the stratification $\sS$ if $h$ has only stratified Morse singularities on the positive dimensional strata of $\sS$, and $h$ is general at all $0$-dimensional strata  (i.e., the differential $dh_x$ at such a point stratum $x$ is not the limit, for $y\to x$, of covectors $\eta_y$ which vanish on the tangent space at $y$ to the stratum containing $y$).  
\end{definition}

 For some linear function $\ell$, we consider the 
 map $(\ell,f) \colon X \to \bC\times \bC$ and its stratified singular locus $\Sing_{\sW}(\ell,f) := \bigcup_{V\in \sW} \Sing (\ell,f)_{|V}$, where $\Sing (\ell,f)_{|V} := \{x\in V \mid \rank \Jac(\ell_{|V},f_{|V}) <2 \}$. This is a closed set, due to the Whitney regularity of the stratification $\sW$.
 
\begin{definition}[Affine polar locus of $f$]\label{d:polarlocus}
 For  a positive dimensional stratum $V\in \sW$,  one defines the polar locus of $f$ of the stratum $V$, relative to $\ell$,  by:
\[ \Gamma_{V} (\ell,f) :=  \overline{\Sing (\ell,f)_{|V} \setminus \Sing f_{|V}} \subset {\overline V}.
 \]
 The algebraic affine set $\Gamma_{\sW} (\ell,f) := \bigcup_{V\in \sW}\Gamma_{V} (\ell,f)$ is called
 the {\it (affine) polar locus} of $f$ relative  to the linear function $\ell$ and to the stratification $\sW$. 
 \end{definition}
Let us also observe that the polar locus $\Gamma_{\sW} (\ell,f)$ contains all strata $V\in \sW$ of dimension 1 which are not included in $\Sing_{\sW}f$. 

In the spirit of Kleiman's transversality theorem \cite{Kl}, one has the following Bertini-Sard type result in the global setting:
 
 \begin{theorem} \cite[Polar Curve Theorem 7.1.2]{Ti-book}\label{t:bertini}
Let $f\colon X\to \bC$ be
an algebraic function on an algebraic subset $X\subset \bC^N$.
There is a Zariski-open connected subset $\Omega_{f}$ of the dual projective space
 $\check\bP^{N-1}$ such that $\ell$ is a Morse function  with respect to the stratification $\sS$, and the polar locus $\Gamma_{\sW} (\ell,f)$  is either a reduced curve for all $\ell\in \Omega_f$, or it is
empty.
\fin
\end{theorem}

\subsection{The topological behavior of linear Morsifications at infinity}

We show here that for a more generic linear function $\ell$ and for $t\neq 0$ with $\vert t \vert$ small enough, the polynomial function $f_t:=f-t\ell$ is not only a polynomial Morse function  on $X$,  but moreover it has  a``good behavior at infinity''.

 
 Let us first define what we mean by good behavior at infinity.
It is well-known that a polynomial function $F\colon \bC^{N} \to \bC$ may fail to be a C$^{\infty}$-fibration even if it has no singularities at all. The simplest and well-known such example, pointed out by Broughton \cite{Bro},  is  $F(x,y)= x+x^{2}y$. It has an empty singular locus,  the fibre over $0$ is isomorphic to $\bC \sqcup \bC^{*}$, whereas all the other fibers are isomorphic to $\bC^{*}$.  In this case one says that $F$ has an  \emph{atypical fibre at infinity}, i.e., the fibre $F^{-1}(0)$ is atypical, and it turns out that it is the only atypical fibre.  Several studies have been conducted in order to understand how to control this phenomenon, see \cite{Ti-book} for details and for extensive bibliography up to the year 2007, and see \cite{ST-polardeg} for one of the most recent  applications.  Here is the  formal definition:

\begin{definition}\label{d:singinfty}
 Given a non-constant polynomial function $f\colon X \to \bC$, we say that the fibre $f^{-1}(\lambda)$ is \emph{not atypical at infinity} if and only if  $\dim [f^{-1}(\lambda) \cap \Sing_{\sW}f ] \le 0$  and there exist a (small) disk $D\subset \bC$ centred at $\lambda$
 and some compact set $K \subset X$ such that the restriction $f_{|}: f^{-1}(D)\cap (X\setminus K) \to \bC$ is a trivial stratified fibration.  
\end{definition}
It is well-known that any polynomial function $f\colon X \to \bC$ can have only a finite number of fibers which are atypical at infinity\footnote{See \cite{Ti-book} for a proof and other details.}, and that they may occur even if  $\Sing_{\sW} f = \emptyset$.

We  show in the following that for any $f$,  its linear Morsification $f_{t} = f - t\ell$ is a polynomial without atypical fibers at infinity provided that $\ell$ is a general enough linear form.  It is worth pointing out here that this property is far from holding for any linear Morsification, as the following example shows: $f_{t}=x^{2}y + tx$ is a linear Morsification of $f=x^{2}y$ for all $t\not= 0$, it has no singularity at all, but 
the fibre $f_{t}^{-1}(0)$ is atypical at infinity for all $t\not= 0$.

 \begin{theorem}\label{t:lefschetz}
Let $X\subset \bC^N$ be an affine variety and let $f\colon X\to \bC$ be a  non-constant polynomial function.  There is a Zariski-open connected subset $\Lambda_{f} \subset \check\bP^{N-1}$ of linear functions $\ell\colon \bC^N\to \bC$ such that, for any $\ell\in \Lambda_{f}$ and any general value $t\in \bC$,  
the polynomial function $f_t: =f-t\ell$ has only Morse $\sW$-stratified singularities on $X$, and has no atypical fibers at infinity.
  \end{theorem}
\begin{proof}
 Let us consider the following graph embedding:
 \[ i: X\hookrightarrow \bC^{N+1},  \ \ (x_{1}, \ldots , x_{N}) \mapsto \big( x_{1}, \ldots , x_{N}, f(x_{1}, \ldots , x_{N})\big)
 \]
 which is an algebraic isomorphism on its image. It is known that the Whitney stratification is invariant under an analytic isomorphisms since the Whitney $(b)$-condition is equivalent with the Kuo-Verdier condition $(w)$, see e.g. \cite{Ve}. Let then $\sZ$ be the Whitney stratification of $i(X)$ which is the image of the Whitney stratification $\sW$ of $X$ by the isomorphism $i$.
 Let $\overline{i(X)} \subset \bP^{N+1}$ be the projective closure of the affine variety $i(X)$, and let $\overline{\sZ}$
 be  a Whitney stratification of  $\overline{i(X)}$ which extends the stratification $\sZ$ such that the intersection  $\overline{i(X)}\cap H^{\infty}$ with the hyperplane at infinity $H^{\infty} \simeq \bP^{N}\subset \bP^{N+1}$ is a union of strata.
 
 We say that a linear form $l\in \check \bP^{N}$ is \emph{general with respect to $i(X)$} iff the projective hyperplane $\{l=0\} \subset H^{\infty}$ is stratified-transversal to  $\overline{i(X)}\cap H^{\infty}$ and {$l\colon \bC^{N+1} \to \bC$} has only Morse stratified singularities on the positive dimensional strata of $i(X)$.  In this case we say that \emph{$l$ defines an affine Lefschetz pencil on  $i(X)$},  i.e., the pencil of affine hyperplanes $\{l=\alpha\}\subset \bC^{N+1}$ for $\alpha\in \bC$. The classical stratified Lefschetz theory  applied to the  stratified affine variety $i(X)$ ensures that there exists a Zariski-open subset $\Lambda \subset \check \bP^{N}$ of linear forms $l \colon \bC^{N+1}\to \bC$ such that $l$ defines a Lefschetz pencil on  $i(X)$.
  Such a  general linear form $l\in \Lambda$ looks like:
 $$l = x_{N+1} - \sum_{i=1}^{N}a_{i} x_{i},$$ 
and the genericity of the coefficients $a_{i} \in \bC$ ensures that the linear form 
$$l_{t} := x_{N+1} - t \sum_{i=1}^{N}a_{i} x_{i}$$
 defines a Lefschetz pencil on  $i(X)$, for all $t\in \bC^{*}$ close enough to 0.  In particular,  $l_{t}$ has no atypical fibres at infinity in the sense of Definition \ref{d:singinfty}.
 
  The pull-back of $l_{t}$ by the isomorphism $i$ is then precisely the polynomial function $f_{t} = f -t\ell$, where $\ell :=\sum_{i=1}^{N}a_{i} x_{i}$. It follows that $f_{t}$ is a linear Morsification and has no atypical fibers at infinity in the sense of Definition \ref{d:singinfty}.  
  The Zariski open subset of these generic linear functions $\ell =\sum_{i=1}^{N}a_{i} x_{i}$ will then be our  $\Lambda_{f}\subset  \check\bP^{N-1}$. 
 \end{proof}
 

\begin{definition}[General linear Morsification]\label{d:linmorsif}
    Let  $\Omega  := \Omega_{f} \cap \Lambda_{f} \subset  \check\bP^{N-1}$ be the Zariski open connected subset of linear functions   $\ell\colon \bC^N \to \bC$ which satisfy the conditions of Theorem \ref{t:bertini} and of Theorem \ref{t:lefschetz}.
     A linear function $\ell\in \Omega$ will be called \emph{general linear function} with respect to $X$, and  
      the linear deformation $f_t:=f- t \ell$ will be called  a \emph{general linear Morsification of $f$}.
\end{definition}

 It follows from the above definition that 
the restrictions  $f_{| \{\ell =b\}\setminus \Sing_{\sW} f}$ and $\ell_{|\{f = a\}}$,  for generic complex numbers $a, b$, are both stratified  Morse functions relative to the induced stratifications on the slices.

 Any general linear Morsification $f_{t}$ of $f$ has thus isolated singularities only. These are either the stratified Morse points of $f_{t}$, or the 0-dimensional strata of $X$ -- which are counted by definition as singularities of any function on $X$. It also turns out from the above definition  (using the connectivity of $\Omega$) that the number of stratified Morse points of $f_{t}$ on each stratum $V\in \sW$ is independent of  $\ell \in \Omega$. 
 
 In case  $X=\bC^{n+1}=\bC^N$, the number of Morse points of a general linear Morsification of $f$ will be denoted by $m_{f}$. We will call it  \emph{the Morse number of $f$}.

\begin{corollary}\label{c:variationtop}
Let  $f \colon \bC^{n+1}\to \bC$ be a non-constant polynomial. Let $f_{t}:= f - t\ell$, for some $\ell \in \Omega$, be a general linear Morsification, and let  $m_{f}$ be  
 the Morse number of $f$.  Let $b_{n}(f_{t})$ and  $b_{n}(f)$ denote the top Betti numbers of the general fibre of $f_{t}$ and of $f$, respectively. 
  Then $m_{f} = b_{n}(f_{t})$ and   $m_{f} \ge b_{n}(f)$.
 \end{corollary}
   
\begin{proof}
For any $t$ in some small enough disk $D\subset \bC$ at the origin, the general fibre  $G_{t}$ of $f_{t}$ is affine of dimension $n$, thus its homology groups are trivial in dimension $\ge n+1$. 
In case $t\not= 0$,  by Theorem \ref{t:lefschetz}, the variation of topology of the fibres of $f_{t}$ is concentrated at the singularities of 
$f_{t}$, which are isolated.  The general fibre $G_{t}$  of $f_{t}$ is then homotopy equivalent to a 
bouquet of spheres of dimension $n$, one for each Morse point, cf \cite{ST-duke}, \cite{Ti-book}. This shows the equality $\mu_{f_{t}}= b_{n}(f_{t})$, where  $\mu_{f_{t}}$ is the total Milnor number of $f_t$, and by definition we have the equality $m_{f} = \mu_{f_{t}}$.

 By \cite[Proposition 2.1]{ST-exch}, in any analytic deformation of a polynomial function $f\colon \bC^{n+1}\to \bC$, and in particular in our linear deformation $f_{t}$ with $f_{0} := f$, the following semi-continuity of the top Betti number holds for the top Betti numbers of the general fibres: $b_{n}(f_{t}) \ge b_{n}(f_{0})$.
By the equality $\mu_{f_{t}} = b_{n}(f_{t})$ shown just before, it then follows that $m_{f} \ge b_{n}(f)$.
\end{proof}

\section{Attractors of Morse trajectories}\label{s:atract}
We address here the problem of finding the trajectories 
of stratified Morse points of $f_{t} = f-t\ell$ and of their limits when $t\to 0$.

\subsection{Ends of Morse trajectories}\label{s:detect}\

Whenever a trajectory of a Morse singularity of $f_{t}$ ends at some point $p\in X$ as $t\to 0$, this point
belongs to $\Sing_{\sW}f$. 
 Let  $\ell\in \Omega$, $V\in \sW$ such that $V\not\subset \Sing_{\sW}f$, and  $p\in \Sing_{\sW}f$.  Let us set:
\begin{equation}\label{eq:morseindex}
 m_{p,V}(f) := \# \{ \mbox{Morse points of  } {f_{t}}_{|V}  \mbox{ which converge to $p$ as } t\to 0\}.
\end{equation}

We aim to determine the subset of  $\Sing_{\sW}f$ of limit points to which  at least one stratified Morse singularity  of the restriction ${f_{t}}_{|V}$ abuts when $t\to 0$, namely:
\[M_{V}(f) := \{ p\in \Sing_{\sW}f \mid m_{p,V}(f) >0\} .\]
  This subset of $\overline{V}$ depends on the  choice of $\ell\in \Omega$  whenever $\Sing f_{|V}$ is positive dimensional. However,  since the Morse points of  ${f_{t}}_{|V}$ depend continuously on the choice of $\ell \in \Omega_f$, which is a Zariski open and connected subset of $\check\bP^{N-1}$,  it turns out that the numbers $m_{p,V}(f)$  do not depend on this choice.
 


Let $\Gamma_{V, p} (\ell,f)$ and  $\Gamma_{\sW, p} (\ell,f)$ denote  the germs at $p$  of  $\Gamma_{V} (\ell,f)$  and of $\Gamma_{\sW} (\ell,f)$, respectively, cf. Definition \ref{d:polarlocus}. 

\begin{definition}[{Polar points of  $\Sing_{\sW}f$}]\label{d:pol}
 Let $\ell\in \Omega$, and let $V\in \sW$, $\dim V >0$. We set:
  $$\Pol_{V}(f) :=  \{ p\in \Sing_{\sW}f  \mid  \Gamma_{V, p} (\ell,f) \not= \emptyset \}$$ 
  and
  \[ \Pol(f) := \bigcup_{V\in \sW } \Pol_{V}(f).
  \]
 We call $\Pol_{V}(f)$ the set of \emph{$V$-polar points of  $f$ relative to $\ell$}. 
  \end{definition}

From the above definitions it follows that $\Pol_{V}(f)$ and $\Pol(f)$ are finite sets.  Let us also remark that all these sets depend on $\ell\in \Omega$.  The following statement shows in particular the finiteness of the 
sets of limit points in a given linear Morsification $f_t=f-t\ell$:

\begin{proposition}\label{t:localind}
Let $V \in \sW$, and let $W\subset \overline{V}\m V$ be a stratum of $\sS$ included in $\Sing_{\sW}f$. Then:
 \begin{equation}\label{eq_main}
 W\cap M_{V}(f) \subset  W\cap \Pol_{V}(f) \subset \Sing(\ell|_{W}).
\end{equation}
\end{proposition}
\begin{proof} 
The Morse points of the restriction of $f_{t} = f-t\ell$ to some stratum $V$ satisfy the equations of the polar curve $\Gamma_{V} (\ell,f)$. More precisely, if there is a trajectory of a stratified Morse point $q(t)\in V$ which abuts to some point $p\in  W\subset \overline{V}\setminus V$, then this trajectory, as a set-germ at $p$, is included in the polar curve $\Gamma_{V, p} (\ell,f)$, and therefore  $p\in M_{V}(f)$.
This implies the first inclusion $W\cap M_{V}(f)\subset W\cap \Pol_{V}(f)$.

The second inclusion $W \cap \Pol_{V}(f) \subset \Sing(\ell|_{W})$ has been shown in the proof of \cite[Proposition 4.2]{MT1}. Here is an outline. Let $W\subset \Sing_{\sW}f$,  be a stratum of $\sS$ such that  $W\subset \overline{V}\m V$. 
Let   $T^*_{\overline{V}} \bC^N$ be the conormal space of $\overline{V}$ in the cotangent bundle $T^*\bC^N$,   i.e., the closure of the conormal bundle $T^*_V \bC^N$ in  $T^*\bC^N$, and let $\pi : T^*_{\overline{V}} \bC^N \to \overline{V}$ denote the projection. 
  If $p\not\in \Sing (\ell_{|W})$ then  $\d \ell \not\in \pi^{-1}(p)$. On the other hand, it is well-known that the Whitney stratification $\sS$ is also Thom (a$_{f}$)-regular at all the strata of the singular locus $\Sing_{\sW}f$, cf  \cite{BMM}, see also \cite{Ti-compo} or \cite[Theorem A1.1.7]{Ti-book}. This means that, for the pair of strata $(W,V)$ as above,  $\d f_{|V}$ is independent of $\d \ell_{|V}$  in the neighbourhood of $p\in W$.
In turn, this implies, by definition,  that the polar locus $\Gamma_{V} (\ell,f) $ is empty in the neighbourhood of $p$. Since this holds for any stratum $V$ as above, our statement is proved.
\end{proof}

%

 
\subsection{Detecting the Morse singularities of $f_{t}$}\

In order to compute the numbers $m_{p,V}(f) $, we need to identify conveniently the trajectories of Morse points stratwise, and here is our key result:
\begin{lemma}\label{t:polardiscrim-new}
 Let $V\in \sW$ be a positive dimensional stratum of $X$, such that $V\not\subset \Sing_{\sW}f$.
 The point $q\in V$  is a Morse singularity  of $f_{t}$ on $V$  if and only if 
    $q\in V\cap \Gamma_{V}(\ell,f)$ such that   $\mult_{q}\Bigl(\Gamma_{V}(\ell,f), f_{t|V}^{-1}\bigl(f_{t}(q)\bigr)\Bigr)=2$.
\end{lemma}
\begin{proof} 
 By Proposition \ref{t:localind},  a trajectory of  Morse points $q=q(t)\in V$  inside $\Gamma_{V}(\ell,f)$ abuts, as $t\to 0$,  to some point $p\in M_{V}(f) \subset  \Pol_{V}(f)$. 
 This point $p$ may be a singular point of the curve $\Gamma_{V}(\ell,f)$, however the points $q(t)\in \Gamma_{V}(\ell,f)$ are nonsingular, for all $t\not= 0$ close enough to 0. 
   
   As $V$ is nonsingular, we may apply the classical study to the 
  local singular fibration at $q$ defined by  the map $(\ell,f_{t})_{|V}$, see e.g. \cite{Ti-israel} and the paper \cite{Le} where  L\^{e} D.T introduced this object. Let us point out that $\ell$ is general  enough with respect to $f_{t}$ at $q$, in the sense that the pair $(\ell,f_{t})_{|V}$ defines an isolated complete intersection singularity  at $q$.  
  The function $f_{t|V}$ has a Morse singularity at $q$, thus its Milnor number is 1. The restriction $\ell_{|V}$ is non-singular  at $q$ and transversal to the polar curve $\Gamma_{V}(\ell,f)$, which implies the equality $\mult_{q}\bigl(\Gamma_{V}(\ell,f), \ell_{|V}^{-1}(\ell(q)\bigr)  =1$.  
The polar formula for the Milnor number \cite{Le}, see also \cite{Ti-bouquet}, \cite{Ti-israel}, takes here the form:
    \[  1 = \mult_{q}\bigl(\Gamma_{V}(\ell,f_{t}), f_{t |V}^{-1}(f(q)\bigr) - \mult_{p}\bigl(\Gamma_{V}(\ell,f_{t}), \ell_{|V}^{-1}(\ell(q)\bigr),
    \]
and we may replace $\Gamma_{V}(\ell,f_{t})$ by  $\Gamma_{V}(\ell,f)$ since they are equal sets.
 It then follows  that $\mult_{q}\bigl(\Gamma_{V}(\ell,f), f_{t|V}^{-1}(f(q)\bigr)= 2$, which proves our claim in one direction. 
 
Conversely, if $f_{t|V}$ has no singularity at $q=q(t)\in \Gamma_{V}(\ell,f)$, since its Milnor number at $q$ is 0,  we obtain, like we explained above, the equality:
  \[  0 = \mult_{q}\bigl(\Gamma_{V}(\ell,f), f_{t |V}^{-1}(f(q)\bigr) - \mult_{q}\bigl(\Gamma_{V}(\ell,f), \ell_{|V}^{-1}(\ell(q)\bigr),
    \]
    which implies  $\mult_{q}\bigl(\Gamma_{V}(\ell,f), f_{t|V}^{-1}(f(q)\bigr)=1$. This completes our proof.
\end{proof}

\subsection{Morse trajectories at infinity and their limit points}\label{s:isolregul}\

\begin{definition}\label{d:infinity}
 Let $\overline{X}\subset \bP^{N}$ denote the projective closure of the affine variety $X\subset \bC^N$, and let $\tilde f$ denote the homogenization of $f$ of degree $d = \deg f$ with the new variable $x_{N+1}$.
The following hypersurface    
 \[ \bX := \cl  \{ [x_{1}; \cdots ;x_{N};x_{N+1}]\in X , [\lambda; \tau] \in \bP^{1} \mid \tau  \tilde f(x) - \lambda x_{N+1}^{d} =0\}\subset 
 \overline{X}\times \bP^{1} \subset \bP^{N} \times \bP^{1}
 \]
 is a completion of the graph of $f$, and the projection $\lambda : \bX \to \bC = \bP^{1}\setminus \{[0;1]\}$ is a proper map which extends our polynomial map $f\colon X\to \bC$.
 Let us consider the embedding $X\hookrightarrow \bX$, $\mathbf x \mapsto ([\mathbf x; 1], f(\mathbf x))$.
Its  image identifies with $\bX \setminus \bX^{\infty}$ where $\bX^{\infty} :=  \bX \cap (H^{\infty}\times\bP^{1})$, and   $H^{\infty}$ denotes the hyperplane at infinity  $\bP^{N}\setminus \bC^{N}$.
 \end{definition}
 Let $(p,\alpha)\in \bX^{\infty}$, where $\alpha \in \bP^{1}$. The particular point $\alpha = [1;0]$ will be viewed as the infinite limit $[1;0] := \lim_{\lVert \bx \rVert \to \infty}f(\bx)$.

 Let  then $\overline{\Gamma_{\sW}(\ell,f)}$ denote the closure of the polar curve in $\bX$.  The following result produces a convenient superset of the set of all limit points  at infinity of the Morse trajectories:
 
 \begin{lemma}\label{l:main2}
Let $f \colon X\to \bC$ be a non-constant polynomial map, and let $\ell \in \Omega$.
A Morse point $q(t)$ of $f_{t}$ which escapes to infinity as $t \to 0$ has a trajectory included in a polar branch of $\overline{\Gamma_{\sW}(\ell,f)}$ and intersects $\bX^{\infty}$  at a certain point of the finite set:
  $$\Pol^{\infty}(f) := \{ (p,\alpha)\in \bX^{\infty} \mid  (p,\alpha)\in  \overline{\Gamma_{\sW}(\ell,f)} \cap \bX^{\infty}\}\subset \bP^{N}\times \bP^{1}.$$  
 \end{lemma}
\begin{proof}
As shown by Lemma \ref{t:polardiscrim-new}, any Morse point of $f_{t}$ is contained in the polar locus $\Gamma_{\sW}(\ell,f)$, so this holds for those which escape to  infinity too.   The finiteness of branches of the polar curve implies the finiteness of the set $\Pol^{\infty}(f)$.
\end{proof}
The above Lemma \ref{l:main2} shows in particular that there is a finite set $M^\infty$ of points at infinity $(p,\alpha)\in \bX^{\infty}$, that we call ``attractors'', to which the stratified Morse points of $f_{t}$ abut when $t\to 0$, and that $\Pol^{\infty}(f)$ is an effectively  computable finite set containing those. Let us point out that these points depend on the choice of $\ell\in \Omega$, as well as the affine attractors do.

In the next section we will show how to compute the Morse index of each point of $\Pol^{\infty}(f)$.
Let us  set first the notation:
   \[ m_{(p,\alpha), V} := \# \biggl\{ \begin{array}{ll} \mbox{ Morse points of } f_{t} \mbox{  on the stratum }  V \in \sW  \mbox{ which }\\ \mbox{ converge to the point }  (p,\alpha) \in \Pol^{\infty}(f) \mbox{  as }  t \to 0.   \end{array} \biggr\}.
   \]

 \section{Morse indices at attractors}\label{s:infinity}

We have seen in Section \ref{s:detect} that there are well defined ``attractive poles'',  namely a finite set of points either on $X$ or at infinity to which all the Morse singularities converge when $t \to 0$. We will find a formula for the number of Morse singularities which converge at each attractive pole.  
 
 As before, we consider a non-constant polynomial function $f\colon X\to \bC$ on an irreducible affine variety $X\subset \bC^{N}$ of dimension $n+1\ge 2$,  endowed with the canonical (coarsest) Whitney stratification $\sW$, and general linear Morsification $f_{t} = f - t \ell$, for $\ell \in \Omega$, cf. Definition \ref{d:linmorsif}.

\subsection{Localization at infinity}
We first focus on the case of a point $p$ at infinity, and consider a local chart at $p$. 

We have denoted by $\tilde f$ the homogenization of $f$ of degree $d = \deg f$ with the new variable $x_{N+1}$, and by $H^{\ity}\subset \bP^{N}$ the hyperplane at infinity. Let $\tilde \ell$ denote the linear function $\ell$  viewed as a function on $\bC^{N+1}$, of an extra mute variable $x_{N+1}$. 

 Without loss of generality (after changing the numbering of the coordinates), let  $(p,\alpha)\in \Pol^{\infty}(f)$ such that $p\in \overline{X} \cap \bP^{N}$ is in the chart $x_{1}\not= 0$.
Let then $\overline{f}$ denote the de-homogenization of $\tilde f$
in variable $x_{1}$, i.e., $\overline{f}(x_{2}, \ldots, x_{N+1}) := \tilde f(1, x_{2}, \ldots, x_{N+1}).$
 Similarly, let $\overline{\ell}$ be the de-homogenization of $\tilde \ell$
in  variable $x_{1}$, i.e., $\overline{\ell} (x_{2}, \ldots, x_{N+1}) := \tilde \ell(1, x_{2}, \ldots, x_{N+1}).$
Let $\overline{\Gamma}_{V,p}$ denote the germ of the polar curve\footnote{In this section we consider the closure in $\bP^N$ of the polar curve, denoted by $\overline{\Gamma_{V}(\ell,f)}\subset \bP^{N}$.} $\overline{\Gamma_{V}(\ell,f)}\subset \bP^{N}$ at the same point $p$.   

Note that the local intersection multiplicity between a hypersurface and a curve is actually computed 
as a sum of intersection multiplicities over each irreducible components of the curve. When we deal with meromorphic functions at $p$,  it will turn out that we need to take into account only the positive terms in this sum (cf. Lemma \ref{l:computeord} below).

Let then $\delta_{\ge 0}$ denote the ``positive selection function"
defined as $\delta_{\ge 0}[k] := \frac12(k+ |k|)$,  i.e., $\delta_{\ge 0}[k] := k$ for $k \geq 0$, and $\delta_{\ge 0}[k] := 0$ for $k \leq 0$.
We will use it to select only the  \emph{positive terms} in a sum.


With these notations, we prove: 

\begin{theorem}\label{t:main2}
Let $f\colon X\to \bC$ be a non-constant polynomial, and let $\ell \in \Omega$. Let $V\in \sW$ be some positive dimensional stratum  such that $f_{|V}$ is not constant.
Then, for any $(p,\alpha)\in (\overline{V}\times \bP^1) \cap \Pol^{\infty}(f)$, we have:
 \begin{equation}\label{eq:main4}
m_{(p,\alpha), V} = 
 \sum_{\overline{\gamma}_{p} \subset \overline{\Gamma_{V}(\ell,f)} \atop \overline{\gamma}_{p} \ni (p,\alpha)} 
\delta_{\ge 0}\bigl[\mult_{0} \bigl( \{\overline{f} - \alpha x_{N+1}^{d} =0\}, \overline{\gamma}_{p}\bigr) -   (d-1)\mult_{p} (H^{\ity}, \overline{\gamma}_{p}) \bigr],
\end{equation}
  where the term $\alpha x_{N+1}^{d}$ occurs only when $\alpha$ is a finite value.
\end{theorem}


We will see in the proof that the formula \eqref{eq:main4} does not depend on the local chart of $\bP^{N}$.
Our formula \eqref{eq:main4} for $m_{(p,\alpha), V}$ also determines
 the sets of attractors:
\[ M^{\infty}_{V}(f) := \bigl\{ (p,\alpha)\in \Pol^{\infty}(f) \mid  m_{(p,\alpha), V} >0 \bigr\} \ \mbox{ and } M^{\infty} := \bigcup_{V\in \sW} M^{\infty}_{V}(f).
\]

\

One also notes that if  $V$ is  a stratum such that $f_{|V} =$ constant, then the 
Morse points of $f_{t}$ on $V$ are precisely the tangency points between $V$ and the levels of $\ell$;
in particular they do not depend of $t\not=0$.




%

\subsection{Proof of Theorem \ref{t:main2}}\

We fix some $\ell\in \Omega$ such that the polynomial $f_{t}(x) = f(x) - t \ell(x)$ is a Morsification of $f$
for all $t$ close enough to $0\in \bC$.

\subsubsection{Polar arcs at infinity}
 By Lemma \ref{t:polardiscrim-new}, the Morse singularities of the function $f_{t}$  occur precisely at the nonsingular points of the polar curve $\Gamma_{\sW}(\ell,f)$ at which the
 fibers of $f_{t}$ are tangent.  
We want to count those Morse singularities  which converge to the  point $(p,\alpha)\in \Pol^{\infty}(f)$ as $t \to 0$.  Let us consider \emph{polar arcs at infinity}, i.e., parametrizations of the branches $\gamma$ of the affine polar curve $\Gamma_{\sW}(\ell,f)$  such that
 $(p,\alpha)\in \overline{\gamma}\cap \bX^{\infty}$. More precisely,  there exists a parametrization of $\gamma$:
 $$\mathbf x_{\gamma}(s) := \left( x_{1}(s), \ldots , x_{N}(s)  \right) $$
  for $s$ in a small enough disk $D\subset \bC$ centered at $0\in \bC$, such that $\lim_{s\to 0}\mathbf x_{\gamma}(s) = p\in \overline{\gamma}\cap H^{\infty}\subset \bP^{N}$, and with  $\lim_{s\to 0}f(\mathbf x_{\gamma}(s)) = \alpha$. 
Each $x_{i}(s)$ is a Laurent series of finite order  in the variable $s$. In case the
 projective hypersurface $\{ \ell =0\}\subset \bP^{N-1}$ does not contain the point $p\in H^{\infty}$, we have:
\begin{equation}\label{eq:order}
\lim_{s\to 0} \lVert \ell(\mathbf x_{\gamma}(s)) \rVert  = \infty.
\end{equation}

For any point at infinity $(p,\alpha)\in \Pol^{\infty}(f)$ and some polar branch $\gamma \subset \Gamma_{V}(\ell,f)$ with $(p,\alpha)\in \overline{\gamma}$,   let us set:
 $$m_{\gamma}(p,\alpha) := \# \left\{ \begin{array}{l} \mbox{points of  tangency of } \gamma \mbox{ with  the fibres of }  f_{t} \\ \mbox{ which converge to } (p,\alpha) \mbox{ as } t\to 0 \end{array} \right\},$$
 
Using the identification provided by Lemma \ref{t:polardiscrim-new},   we get the following sum decomposition of the Morse index:
 $$ m_{(p,\alpha), V} = \sum_{(p,\alpha)\in \overline{\gamma} \subset \overline{\Gamma_{V}(\ell,f)} } m_{\gamma}(p,\alpha).$$

 By convention,  $\ord_{0} (f_{\alpha})_{\mid \gamma}$ will denote the order at $0$ of the Laurent series $f _{\mid \gamma}(s) - \lim_{s\to 0}f _{\mid \gamma}(s)$ in case the limit is finite, and will denote $\ord_{0} (f)_{\mid \gamma}$ in case the limit $\lim_{s\to 0}f _{\mid \gamma}(s)$ is infinite.
 \begin{proposition}\label{p:main2}
 If $f_{|V}\not=$ constant, then
 the number $m_{\gamma}(p,\alpha)$ is equal to:
\begin{equation}\label{eq:main3}
\ord_{0} (f_{\alpha} )_{\mid \gamma} -  \ord_{0}\ell_{\mid \gamma}
\end{equation}
whenever this expression is $\ge 0$, and $m_{\gamma}(p,\alpha)= 0$ whenever the expression \eqref{eq:main3} is negative.
\end{proposition}

\begin{proof}
Since $\mathbf x_{\gamma}(s)$ is a Laurent series of finite order, so are $f_{\mid \gamma}$ and $\ell_{\mid \gamma}$.
 We need the following simple but key lemma.
 \begin{lemma}\label{l:computeord}
 Let $g, h: D \to \bC$ be Laurent series of finite order defined on some small disk $D\subset \bC$ centered at 0.
 Let $\nu_{G_{t}(s)}$ denote the number of non-zero solutions of the equation $G_{t}(s):= g(s) - t h(s) =0$  which converge to $0$ whenever $t \to 0$, counting multiplicities.
  We have: 
\begin{enumerate}
  \rm \item \it If  $\ord_{0} g(s) \ge \ord_{0} h(s)$ then $\nu_{G_{t}(s)} = \ord_{0} g(s) - \ord_{0} h(s)$.
 \rm \item \it  If  $\ord_{0} g(s) < \ord_{0} h(s)$ then $\nu_{G_{t}(s)} = 0$.
\end{enumerate}
\end{lemma}
\begin{proof}
 Let   $v:= -\ord_{0}h(s)$ and $r:= - \ord_{0}g(s)$.  Then the series $s^{v}h(s)$ and $s^{r}g(s)$ are both holomorphic and invertible at $0$. For $t\not=0$, the equation $G_{t}(s)=0$ is equivalent to the equation:
\begin{equation}\label{eq:sol}
 \frac{s^{r}g(s)}{s^{v}h(s)}\cdot s^{v-r} = t,
\end{equation}
where the fraction in the first factor is an invertible holomorphic function.
 
\smallskip
\noindent
 (a). If $v-r\ge 0$, the number of solutions $s(t)$ of the equation \eqref{eq:sol} which converge to 0 as $t\to 0$ is precisely $v-r$.

\noindent (b). In  case $v-r < 0$, when $s\to 0$, the left hand side of the equation \eqref{eq:sol}  converges to $+\infty$ in absolute value,
whereas on the right hand side $t\to 0$, which is a contradiction.
\end{proof}

\subsubsection{Computation of $m_{\gamma}(p,\alpha)$}

Let $\gamma$ be a branch of $\Gamma_{\sW}(\ell,f)$,  for some general $\ell \in \Omega$ $\ell(x) := \sum_{i}a_{i}x_{i}$, where $\mathbf a := (a_{1}, \ldots , a_{N})\in \bC^{N}$. 
 Let $(p,\alpha)\in \overline{\gamma} \subset \overline{\Gamma_{V}(\ell,f)}$.


The tangency condition in the above definition of $m_{\gamma}(p,\alpha)$ reads:
\begin{equation}\label{eq:tangcond}
  \biggl \langle  \bigl(\overline{\grad f} - \overline{t \mathbf a}\bigr) (\mathbf x_{\gamma}(s)) ,  \mathbf x'_{\gamma}(s) \biggr \rangle=0,
\end{equation}
 where $\overline{\ \cdot \ }$ means here complex conjugation.
 
What $m_{\gamma}(p,\alpha)$  counts is precisely the number of non-zero solutions $s=s_{i}$ of \eqref{eq:tangcond} which converge to 0 when $t \to 0$. 
%

Our equation \eqref{eq:tangcond} reads as:
\begin{equation}\label{eq:tangcond2}
 \biggl \langle  (\overline{\grad f} - \overline{t \mathbf a})(\mathbf x_{\gamma}(s)) ,  \mathbf x'_{\gamma}(s) \biggr \rangle=  \biggl \langle  \overline{\grad f}(\mathbf x_{\gamma}(s))  ,  \mathbf x'_{\gamma}(s) \biggr \rangle - t \sum_{i}a_{i}\mathbf x'_{\gamma, i}(s) =0.
\end{equation}

By using  Lemma \ref{l:computeord} we get that 
the number of solutions of  equation \eqref{eq:tangcond2} in variable $s$ which converge to $0$ as $t \to 0$ is equal to:
$$\ord_{0} \biggl \langle  \overline{\grad f}(\mathbf x_{\gamma}(s))  ,  \mathbf x'_{\gamma}(s) \biggr \rangle - \ord_{0} \sum_{i}a_{i}\mathbf x'_{\gamma,i}(s).$$

We have $\ord_{0} \sum_{i}a_{i}\mathbf x'_{\gamma,i}(s) = \ord_{0}\ell(\mathbf x_{\gamma}(s)) -1$ since $\sum_{i}a_{i}\mathbf x'_{\gamma,i}(s)$ is the derivative of $\ell(\mathbf x_{\gamma}(s))$. 
Since $\biggl \langle  \overline{\grad f}(\mathbf x_{\gamma}(s))  ,  \mathbf x'_{\gamma}(s) \biggr \rangle$ is the derivative of $f(\mathbf x_{\gamma}(s))$, its order is equal either to $\ord_{0}\bigl[f(\mathbf x_{\gamma}(s))- \lim_{s\to 0}f(\mathbf x_{\gamma}(s))\bigr] -1$ in case the limit $\lim_{s\to 0}f(\mathbf x_{\gamma}(s))$ is finite, 
or it is equal to $\ord_{0}f(\mathbf x_{\gamma}(s)) -1$ in case this limit is infinite.

Thus, after cancelling $+1$ against $-1$, by using the convention of notation $f_{\alpha}$, we get:
\[ m_{\gamma}(p,\alpha) = \ord_{0}f_{\alpha}(\mathbf x_{\gamma}(s))  - \ord_{0}\ell(\mathbf x_{\gamma}(s)).
\]
  This ends the proof of Proposition \ref{p:main2}.
\end{proof}

To prove Theorem \ref{t:main2} we use the computations from the proof of  Proposition \ref{p:main2} and pass to the coordinates at infinity. Along the way we will observe that changing the chart does not affect the results.
Our polynomial function $f \colon X\to \bC$  is the representative 
in the chart $x_{N+1}=1$ of the restriction to $\overline{X}$ of the rational function 
$$\frac{\tilde f(x_{1}, \ldots , x_{N+1})}{x_{N+1}^{d}} : \bP^{N} \dashrightarrow \bP^{1}.$$

%
  Without loss of generality, let us assume that the point $p$ is in the chart $x_{1}\not=0$, and more precisely $p := [1;0 ; \cdots ; 0]$.
  
  In the new chart,  our expression \eqref{eq:main3} transforms as follows: the polar branch $\overline{\gamma}$ becomes
  a branch denoted $\overline{\gamma}_{p}$ of the germ at $p$ of the polar locus, denoted $\overline{\Gamma}_{V,p}$, where we consider a Puiseux parametrisation of $\overline{\gamma}_{p}$ at the origin of the new coordinate system. We then get:
\begin{multline} 
 \ord_{0} (f_{\alpha} )_{\mid \gamma} = \ord_{0}\biggl( \frac{\overline{f}(x_{2}, \ldots , x_{N+1}) }{x_{N+1}^{d}} 
 -\alpha
 \biggr)_{\mid \overline{\gamma}_{p}} = \ \ \ \ \ \ \ \ \ \ \ \ \ \ \ \ \ \  \nonumber   \\    \ \ \ \ \ \ \ \ \  
 \mult_{0} \bigl( \{\overline{f} -  \alpha x_{N+1}^{d}  =0\}, \overline{\gamma}_{p}\bigr) - d \cdot \mult_{0}\bigl(\{x_{N+1}=0\}, \overline{\gamma}_{p}\bigr) 
\end{multline}

and 
$$  \ord_{0}\ell_{\mid \gamma}= \ord_{0}\biggl( \frac{\overline{\ell}(x_{2}, \ldots , x_{N+1})}{x_{N+1}} 
 \biggr)_{\mid \overline{\gamma}_{p}} = \mult_{0} \bigl( \{\overline{\ell}=0\}, \overline{\gamma}_{p}\bigr) -  \mult_{0}\bigl(\{x_{N+1}=0\}, \overline{\gamma}_{p}\bigr).
 $$
Note that, according to our convention on the notation $f_{\alpha}$,  in the above formula $\alpha$ occurs only if it is a finite value.

 We now remark that by our generic choice of $\ell$, the hyperplane $\{\overline{\ell}=0\}$ does not contain the point $p$, and thus $\mult_{0} \bigl( \{\overline{\ell}=0\}, \overline{\gamma}_{p}\bigr) =0$.
 
 Then, by applying again the key Lemma \ref{t:polardiscrim-new},  we obtain that  $ m_{\gamma}(p,\alpha)$ equals the expression:
  $$\mult_{0} \bigl( \{\overline{f} -\alpha x_{N+1}^{d} =0\}
  -   (d -1) \cdot \mult_{0}\bigl(x_{N+1}=0\}, \overline{\gamma}_{p}\bigr) 
  $$
  in case this is $\ge 0$, or we get that $m_{\gamma}(p,\alpha)=0$ if this expression is $\le 0$.
  
By taking now the sum over $\overline{\gamma}_{p}$, we then get our formula \eqref{eq:main4}.
This ends the proof of Theorem \ref{t:main2}.
 
 \subsection{The Morse index at an affine attractor}
 
We derive from the proof of Theorem \ref{t:main2} the corresponding formula for the Morse index $m_{p, V}$ in case $p\in X$, which also determines the set of attractors $M_{V}(f)$:

\begin{corollary}\label{c:main2}
Let  $p\in M_{V}(f)\subset X$ be an affine attractor of Morse points of $f_{t}$, for some positive dimensional stratum $V\in \sW$. We have:
\begin{equation}\label{eq:main1}
  m_{p,V} = \mult_{p}\left(\Gamma_{V,p}(\ell,f), f^{-1}(f(p))\right) - \mult_{p}\left(\Gamma_{V,p}(\ell,f), \ell^{-1}(\ell(p))\right).
\end{equation}
 \end{corollary}

Formula \eqref{eq:main1} has been proved recently for $f$ with isolated singularity on $X$, in the case $V$ is the maximal stratum $X_{\reg}$ in \cite{MT1}, and for any stratum $V\in \sW$ in  \cite{MT2}.

\begin{proof}
Let $p\in \Pol_{V}(f)$, and let  $\gamma \subset \Gamma_{V}(\ell,f)$ be  some local polar branch $\gamma \subset \Gamma_{V}(\ell,f)$ at $p$,
parametrised as $\gamma(s)$ for $s$ is a small disk $D$ centred as 0 and such that  $\gamma(0)=p$.
  Let us set:
 $$m_{\gamma}(p) := \# \left\{ \begin{array}{l} \mbox{points of  tangency of } \gamma \mbox{ with  the fibres of }  f_{t} \\ \mbox{ which converge to } p \mbox{ as } t\to 0 \end{array} \right\}.$$
 
 In our setting, since the germ of $f$ at $p$ is in the square of the maximal ideal of the local algebra $\cO_{X,p}$,
one has\footnote{See e.g. \cite[Proposition 1.2]{Le}, \cite[Corollary 2.3]{Ti-ens}.} the inequality $\ord_{0} f _{\mid \gamma} \ge \ord_{0}\ell_{\mid \gamma}$.

 Using Lemma \ref{t:polardiscrim-new} in the same way as in the proof of Theorem \ref{t:main2}, we obtain the
 following formula which is analogous to \eqref{eq:main3} in Proposition \ref{p:main2}:
  \begin{equation}\label{eq:main5}
m_{\gamma}(p) = \ord_{0} f _{\mid \gamma} -  \ord_{0}\ell_{\mid \gamma} \ge 0.
\end{equation}
  In contrast,  in the setting ``at infinity'' of \eqref{eq:main3},  the above difference
could have been negative for  some polar branches. This behaviour can be seen in  Examples \ref{ex:lau-s} and \ref{ex:lau-s-more} below.

Lastly, the following equalities hold by definition:  
$$ \mult_{p}\left(\Gamma_{V,p}(\ell,f), f^{-1}(f(p))\right)  = \sum_{\gamma \subset\Gamma_{V,p}}\ord_{0} f _{\mid \gamma},$$
$$\mult_{p}\left(\Gamma_{V,p}(\ell,f), \ell^{-1}(\ell(p))\right)  = \sum_{\gamma  \subset\Gamma_{V,p}} \ord_{0}\ell_{\mid \gamma},$$
and also:
$$m_{p,V} = \sum_{\gamma \subset \Gamma_{V,p}}m_{\gamma}(p).$$
Altogether, we obtain our formula \eqref{eq:main1} by summing up \eqref{eq:main5} over all branches $\gamma$.
\end{proof}
 

\section{Examples}

\subsection{Example}\label{ex:classic}
 Let $X =\bC^{2}$,  $f:=x+x^{2} y$, and $\ell := x+y$.  
 
 The stratifications $\sW$ and $\sS$ have a single stratum $\bC^{2}$, since  $\Sing f = \emptyset$. 
  The general linear Morsification $f_{t}= f-t(x+y)$ has two Morse points\footnote{As one can easily check by direct computations.} which go to infinity when $t\to 0$. By Lemma \ref{l:main2},
  the trajectories of these Morse points abut to  the set of ``point attractors'' included in $\Pol^{\infty}(f)$.  By computing the polar curve $\Gamma(\ell, f)$, one easily finds that  $\Pol^{\infty}(f) = \{(p,0), (q,[1;0])\}$, where $p = [0;1;0]\in \bP^{2}\cap H^{\infty}$, and $q= [2;1;0]\in \bP^{2}\cap H^{\infty}$.
     Let us compute the Morse indices at those two points.
 
  The polar curve $\Gamma(\ell, f) = \{ 2xy +1 - x^{2} =0\} \subset \bC^{2}$ has a unique branch $\gamma =\Gamma(\ell, f)$ at $p$, and a parametrisation of it is $x(s) =s$, $y(s) = \frac{s}2 - \frac1{2s}$. We then have 
  $$ \lim_{s\to 0}f(x(s), y(s)) = 0, \ \ \ \lim_{s\to 0} \vert\vert \ell(x(s), y(s)) \vert\vert = \infty,$$
We get:
    \[ \ord_{0}(\ell_{\mid \gamma}) = \ord_{0}\left( s + \frac{s}2 - \frac1{2s} \right) = - 1\]
    and
  \[ \ord_{0}(f_{\mid \gamma}) = \ord_{0}\left( s + s^{2}\left(\frac{s}2 - \frac1{2s} \right) \right) = \ord_{0}\left(s + \frac{s^{3}}2 - \frac{s}2\right) = 1.\]
By applying Proposition  \ref{p:main2}, we obtain:
 \[ m_{\gamma}(p,0) = 1- (-1) = 2,
 \]
 which tells  that precisely two Morse points of $f_{t}$ go to infinity asymptotically to the fibre $f^{-1}(0)$ and converge to the point $p = [0;1;0]\in H^{\infty}$ as $t \to 0$.
 
 Let us also verify Theorem \ref{t:main2}.  Changing the chart of $\bP^{2}$ to $y=1$ with the point $p$ as the origin,  we get the  following parametrisation of the polar branch $\overline{\gamma}_{p} : z(s) = s, x(s) = \omega s^{2} +\hot$, for some coefficient $\omega\in \bC$,  where, here and later, $\hot$ stands for higher order terms in the variable $s$. Then:
 $$ \mult_{0}(\{xz^{2} +x^{2}=0\}, \overline{\gamma}_{p}) = 4, \ \ \mult_{0}(\{z=0\}, \overline{\gamma}_{p})  =1,
 $$
 thus formula \eqref{eq:main4} yields: $m_{\gamma}(p,0) = 4 - 2\cdot 1 = 2$, which confirms the result already obtained above.
 
 Let us now compute the Morse index at $(q, [1;0])$. In the chart $y=1$ the point $q$ becomes $(2,0)$, and we  translate it to become the origin $(0,0)$. In the new coordinates, the equation of $\overline{f}_{q}$ is then $(2-x) z^{2}+ (2-x)^{2}=0$. That of the polar branch germ $\overline{\gamma_{q}}$ is: $x(2-x) +z^{2} =0$ and has parametrisation $z(s) = s, x(s) = -\frac12 s^{2} +\hot$ at $(0,0)$.
 We get:
 $$ \mult_{0}\bigl(\{(2-x) z^{2} + (2-x)^{2}=0\}, \overline{\gamma_{q}}\bigr) = 0 \ \mbox{ and }  \mult_{0}(\{z=0\}, \overline{\gamma_{q}})  =1.
 $$
Formula \eqref{eq:main4} then reads: $0-2 <0$, in which case we have $m_{\gamma}(q,[1;0]) =0$,  confirming that  there are no Morse points of $f_{t}$ which abut to the point $q$ when $t\to 0$.

\subsection{Example.}\label{ex:lau-s}
Let  $f: \bC^{2}\to \bC$,  $f:=xy+\frac13 x^{3} y^{2}$, and $\ell := x+y$. This example\footnote{Constructed with the help of Lauren\c tiu P\u aunescu.} illustrates the following phenomenon (which  does not occur in the preceding example):  the general linear Morsification $f_{t} = f-t\ell$ has a Morse point which escapes to infinity when $t\to\ity$, and along its trajectory the value of $f$ becomes infinite. 

We have $\Sing f = \{(0,0)\}$,  and the stratification $\sS$ has one single stratum of positive dimension, namely $\bC^{2}\setminus \{(0,0)\}$. 
The polar curve is $\Gamma(\ell, f) = \{ y -x + x^{2}y^{2} - \frac23 x^{3}y =0\} \subset \bC^{2}$.
The intersection $\overline{\Gamma(\ell, f)}\cap H^{\ity}$ consists of the points  
$p = [1;0;0]$, $q = [0;1;0]$, and $w = [3;2;0]$. 

\subsubsection{Computations at $\cO :=(0,0)$.}
The branch at $\cO$ of the curve $\Gamma(\ell, f)$ is parametrised as $\gamma: x(s) = s, y(s) = s +\hot$\\
We obtain $\ord_{0} f(\gamma(s)) = 2$ and $\ord_{0} \ell(\gamma(s)) = 1$, and by applying Corollary \ref{c:main2} we get the Morse index $m_{\cO} = 2-1 =1$, which confirms that the Morse singularity of $f$ deforms into a Morse singularity of $f_{t}$.

\subsubsection{Computations at $q$.}
The branch of $\Gamma(\ell, f)$ at $q$ is parametrised as $\gamma: x(s) = s, y(s) = - \frac{1}{s^{2}} +\hot$\\
We obtain $\ord_{0} f(\gamma(s)) = -1$ and $\ord_{0} \ell(\gamma(s)) = -2$, and $\lim_{s\to 0} \Vert f(\gamma(s)) \Vert = \infty$. By applying Proposition \ref{p:main2} we get the Morse index $m_{q} = (-1) - (-2) =1$, which means that one Morse singularity of $f_{t}$ abuts to the point $q\in H^{\ity}$ when $t\to 0$, whereas the limit of $f$ along this trajectory is infinite.

\subsubsection{Computations at $p$.}
The branch of $\Gamma(\ell, f)$ at the point $p$ is parametrised as $\gamma: x(s) = \frac1s, y(s) = - \frac12 s^{2} +\hot$\\
We obtain $\ord_{0} f(\gamma(s)) = 1$ and $\ord_{0} \ell(\gamma(s)) = -1$, and $\lim_{s\to 0}  f(\gamma(s))  = 0$. By applying Proposition \ref{p:main2} we get the Morse index $m_{p} = 1 - (-1) =2$, which means that two Morse singularities of $f_{t}$ abut to the point $p\in H^{\ity}$ when $t\to 0$,  and their trajectories are asymptotic to the fibre $f^{-1}(0)$.

\subsubsection{Computations at $w$.}
At the point $w$, the branch of $\Gamma(\ell, f)$  is parametrised as $\gamma: x(s) = \frac1s, y(s) = -  \frac2{3s} +\hot$\\
We obtain $\ord_{0} f(\gamma(s)) = -2$ and $\ord_{0} \ell(\gamma(s)) = -1$. The difference is $-2 - (-1) <0$, and in this case, by Proposition \ref{p:main2}, the Morse index $m_{w}$ equals 0. This means that no Morse 
 singularity of $f_{t}$ abuts to the point $w\in H^{\ity}$ at infinity when $t\to 0$.

\noindent Conclusions: 
\\
\noindent (1).  The limit points of the trajectories of the four Morse points of $f_{t}$
are $\cO$, $p$ and $q$. \\
\noindent (2). 
The corresponding  Morse indices of $f$ at these points are $m_{\cO} =1$, $m_{q}  =1$, $m_{p} =2$, respectively.
 
\subsection{Example.}\label{ex:lau-s-more}
 We modify the preceding example as $f:=xy+\frac13 x^{3} y^{2}+ x^{6}$, with $\ell := x+y$ still as general linear function.
The singular locus $\Sing f$ consists of 8 points (all being Morse singularities). With similar computations as above, we find that the general linear Morsification $f_{t}= f-t\ell$ has 9 Morse points. When $t\to 0$,  one Morse point escapes to infinity, namely it goes to the point $q = [0;1;0]$,  but its trajectory is not asymptotic to any fibre of $f$.

Since the polar curve intersection with the line at infinity $\overline{\Gamma(\ell, f)}\cap H^{\ity}$ is the single point $q = [0;1;0]$ without being asymptotic to any fibre of $f$, this time $f$ has no ``singularities at infinity'' according to the criteria\footnote{Due to the absence of singularities at infinity, the Milnor number $\mu_{f}=8$  also equals the top Betti number of the generic fibre of $f$.} proved in \cite{ST-duke}, see also \cite{Ti-book}.




\begin{thebibliography}{10001}

\bibitem[BMM]{BMM}
J. Brian\c{c}on, Ph. Maisonobe, M. Merle,  {\it Localisation de 
syst\`emes diff\'erentiels, stratifications de Whitney et condition de Thom},   
\newblock {Inventiones Math.} 117,  531--550, 1994.

\bibitem[Br]{Bri} E.~Brieskorn,
 {\it Die Monodromie der isolierten Singularit\"{a}ten von  Hyperfl\"{a}chen},
 Manuscripta Math. (1970), no. 2, 103--161.
 
 \bibitem[Bro]{Bro} S.~A. Broughton,
{\it Milnor numbers and the topology of polynomial hypersurfaces},
 Invent. Math. 92 (1988),  no. 2, 217--241.
 
 \bibitem[DHOST]{DHOST} J. Draisma, E. Horobe\c{t}, G. Ottaviani, B. Sturmfels, R. Thomas, {\it The Euclidean distance degree of an algebraic variety}, Found. Comput. Math. 16 (2016), no. 1, 99--149.
 
 \bibitem[GM]{GM} 
 M. Goresky, R. MacPherson, {\it Stratified Morse theory}, Ergebnisse der Mathematik und ihrer Grenzgebiete (3), 14. Springer-Verlag, Berlin, 1988.


\bibitem[Kl]{Kl} 
S.~Kleiman, {\it The transversality of a general translate}, Compositio Math. 28 (1974), 287--297.

\bibitem[Le]{Le}
L\^{e} D.T., {\it Calcul du nombre de cycles \'evanouissants d'une hypersurface complexe},
Ann. Inst. Fourier (Grenoble) 23 (1973), no. 4, 261--270. 



\bibitem[MRW1]{MRW1} L. Maxim, J.I. Rodriguez, B. Wang, {\it Euclidean distance degree of the multiview variety}, SIAM J. Appl. Algebra Geometry 4 (2020), no. 1, 28--48.


\bibitem[MRW2]{MRW2} L. Maxim, J.I. Rodriguez, B. Wang, {\it A Morse theoretic approach to non-isolated singularities and applications to optimization},
J. Pure Appl. Algebra 226 (2022), Issue 3, 106865.

\bibitem[MT1]{MT1} L. Maxim, M.~Tib\u{a}r,  {\it Euclidean distance degree and limit points in a Morsification},   Adv. in Appl. Math. 152 (2024), Paper No. 102597, 20 pp.

\bibitem[MT2]{MT2} L. Maxim, M.~Tib\u{a}r,  {\it Morse numbers of function germs with isolated singularities},   Q. J. Math. 74 (2023), no. 4, 1535--1544.

\bibitem[Mi]{Mi}
J. W.~Milnor,  {\it Singular points of complex hypersurfaces},  Annals of Mathematics Studies, No. 61, Princeton University Press, Princeton, N.J.; University of Tokyo Press, Tokyo, 1968.


\bibitem[ST1]{ST-duke} 
D. Siersma, M. Tib\u ar, {\it Singularities at infinity
 and their vanishing cycles}, Duke Math. J.  80 (1995), no. 3, 771--783.

\bibitem[ST2]{ST-defo} 
D. Siersma, M. Tib\u ar,  {\it Deformations of polynomials, boundary singularities and monodromy},
 Moscow Math. J. 3 (2003),  no. 2, 661--679.

\bibitem[ST3]{ST-exch} 
D. Siersma, M. Tib\u ar, {\it Singularity Exchange at the Frontier of the Space},  in:  Real and Complex Singularities, 
 Trends in Mathematics,  pp. 327--342, Birkh\"auser, Basel, 2006.
 
\bibitem[ST4]{ST-betti}
 D. Siersma, M. Tib\u ar, {\it Betti bounds of polynomials},
Mosc. Math. J. 11 (2011), no. 3, 599--615.

\bibitem[ST5]{ST-polardeg} 
D. Siersma, M. Tib\u ar, {\it Polar degree and vanishing cycles},  J. Topol.  15 (2022), no. 4, 1807--1832.

\bibitem[Ti1]{Ti-ens} M. Tib\u ar,
 {\it Carrousel monodromy and Lefschetz number of singularities},
  Enseign.Math. (2)  39 (1993), no. 3-4, 233--247.
  
\bibitem[Ti2]{Ti-bouquet}
M. Tib\u ar,  {\it Bouquet decomposition of the Milnor fibre}, Topology 35 (1996), no. 1, 227--241. 


\bibitem[Ti3]{Ti-compo} M. Tib\u ar,   {\it Topology at infinity of polynomial mappings and Thom condition}.
   Compositio Math. 111 (1998), 89--109.

\bibitem[Ti4]{Ti-book}
M.~Tib\u{a}r,  {\it Polynomials and vanishing cycles},
Cambridge Tracts in Mathematics, vol. 170.  Cambridge University Press, Cambridge, 2007.

\bibitem[Ti5]{Ti-israel} 
M. Tib\u ar, {\it The vanishing neighbourhood of non-isolated singularities}, Israel J. Math. 157 (2007), 309--322.

\bibitem[Ve]{Ve}
  J.-L. Verdier, {\it Stratifications de Whitney et th\' eor\`eme de
Bertini-Sard}, Invent. Math. 36 (1976), 295--312.





\end{thebibliography}
\end{document}